\documentclass[12 pt]{amsart}
\usepackage{amsmath}
\usepackage{amsthm}
\usepackage{amsfonts}
\usepackage{amssymb}
\usepackage{eucal}
\usepackage{graphicx}
\usepackage[all,knot]{xy}
\usepackage{color}
\usepackage[all,knot]{xy}
\xyoption{arc}

\newcommand{\CC}{\mathcal{C}}
\newcommand{\YD}{\mathcal{YD}}

\newcommand\id{\operatorname{I}}

\newcommand{\End}{{\rm End}}
\newcommand{\Rep}{{\rm Rep}}

\newcommand{\ot}{\otimes}
\newcommand{\B}{\mathcal{B}}

\newcommand{\e}{\mathbf{e}}

\newcommand{\lan}{\langle}
\newcommand{\ra}{\rangle}

\newcommand{\Z}{\mathbb{Z}}
\newcommand{\Zz}{\mathcal{Z}}

\newcommand{\N}{\mathbb{N}}
\newcommand{\Q}{\mathbb{Q}}
\newcommand{\C}{\mathbb{C}}

\newcommand{\supp}{\mathrm{supp}}
\newcommand{\Aut}{\mathrm{Aut}}
\newcommand{\Inf}{\mathrm{Inf}}

\newtheorem{theorem}{Theorem}[section]
\newtheorem{lemma}[theorem]{Lemma}
\newtheorem{prop}[theorem]{Proposition}
\theoremstyle{definition}
\newtheorem{definition}[theorem]{Definition}
\newtheorem{example}[theorem]{Example}

\newtheorem{defn}[theorem]{Definition}

\newtheorem{conj}[theorem]{Conjecture}
\newtheorem{cor}[theorem]{Corollary}
\theoremstyle{remark}
\newtheorem{remark}[theorem]{Remark}

\numberwithin{equation}{section}

\newcounter{commentcounter}

\begin{document}
\title{Braid Representations from Unitary Braided Vector Spaces}
\author{C\'esar Galindo}
\address{Departamento de Matem\'aticas\\ Universidad de los Andes\\ Carrera 1 N.
18A - 10,  Bogot\'a, Colombia}
\email{cn.galindo1116@uniandes.edu.co}
\author{Eric C. Rowell}
\address{Department of Mathematics\\
    Texas A\&M University \\
    College Station, TX 77843-3368
}
\email{rowell@math.tamu.edu}
\thanks{ C.G. would like to thank the hospitality of the Mathematics Departments at MIT and Texas A\&M where part of this work was carried out. C.G. was partially supported by Vicerrector\'{i}a de Investigaciones de
la Universidad de los Andes. E.C.R. gratefully acknowledges the partial support of US NSF grant DMS-1108725 and the hospitality and support of BICMR, Peking University where part of this article was prepared. }

\begin{abstract}
 We investigate braid group representations associated with unitary braided vector spaces, focusing on a conjecture that such representations should have virtually abelian images in general and finite image provided the braiding has finite order.  We verify this conjecture for the two infinite families of Gaussian and group-type braided vector spaces, as well as the generalization to quasi-braided vector spaces of group-type.
\end{abstract}
\maketitle
\section{Introduction}

 In this article we study certain unitary representations of braid group $\B_n$ generated by
$\sigma_1,\ldots,\sigma_{n-1}$ satisfying:
\begin{enumerate}
 \item[(R1)] $\sigma_i\sigma_{i+1}\sigma_i=\sigma_{i+1}\sigma_i\sigma_{i+1}$
\item[(R2)] $\sigma_i\sigma_j=\sigma_j\sigma_i$ for $|i-j|>1$.
\end{enumerate}

Sequences of unitary representations of $\B_n$ are of central importance in the topological model for
quantum computation, in which the computational power is intertwined
with the sizes of the closed images of the unitary braid group action on the state spaces of $n$-punctured disks.  In this context the images of the braid group generators $\sigma_i$ serve as quantum gates, and the goal is to determine when braiding alone is universal (\emph{i.e.}, any unitary operator can be efficiently approximated as the image of a braid, within a given error threshold).  For universal braiding it is essentially sufficient to show that the closed images of the braid group $\B_n$ on each irreducible sector $V$ contains $SU(V)$, for $n$ large enough.

The unitary braid representations in this setting are described algebraically through unitary braided fusion categories.  An object $X$ in such a category $\CC$ gives rise, through the braiding operators $c_{X,X}\in\End_\CC(X^{\ot 2})$, to unitary braid group representations $\rho_X^n$ on $\End_\CC(X^{\ot n})$. Geometrically, this corresponds to disks with $n$-punctures, each labeled by the object $X$ (with boundary label varying over all simple objects).

Recent articles devoted to determining
images of unitary braid group representations include: \cite{FLW,LRW,LR,RUMA} while Jones \cite{Jones86} seems to have been the first to consider the problem of characterizing the images.  In \cite{FLW,LRW} the cases where the braid group images are imprimitive, decomposable or finite are eliminated and then
the remaining possibilities are analyzed in terms of the eigenvalues of the generators $\sigma_i$, using Lie theory.

Braid group representations emerge in quantum information in another context through the Yang-Baxter equation (\cite{KL}), as entanglement resources.  Recall that $c\in\Aut(V^{\ot 2})$ satisfies the Yang-Baxter equation if
$$(c\ot I_V)(I_V\ot c)(c\ot I_V)=(I_V\ot c)(c\ot I_V)(I_V\ot c).$$
In this case the pair $(V,c)$ will be called a \textbf{braided vectors space} (BVS).
Given a BVS, $(V,c)$, we obtain braid group representations
$\rho^c:\B_n\rightarrow GL(V^{\otimes n})$
defined on generators by $\rho^c(\sigma_i)=I^{\ot (i-1)}\ot c\ot I^{\ot
(n-i-1)}$.

A connection between braid group representations in these two quantum computational contexts was established and generalized in \cite{RW,GHR} using the concept of \emph{localization}.  Roughly, a localization of the sequence of representations $\rho_X^n$ is a BVS $(V,c)$ that faithfully encodes $\rho_X^n$, uniformly for all $n$.  It is conjectured in \emph{loc. cit.} that the sequence of representations $\rho_X^n$ are localizable if, and only if, $\dim(X)^2\in\N$.  On the other hand, the property $F$ conjecture (see \cite{NR}) predicts that the braid group representations $\rho_X^n$ have finite image if, and only if, $\dim(X)^2\in\N$.

Motivated by the above we consider the problem of characterizing the images of the braid group representation associated with a unitary BVS.  Recall that a \textbf{virtually abelian} group is a group with an abelian subgroup of finite index. We make the following:
\begin{conj}\label{mainconj}
 Suppose $(V,c)$ is a unitary BVS.
 \begin{enumerate}
  \item[(a)] Then $\rho^c(\B_n)$ is a virtually abelian group for all
$n$.
\item[(b)]  If in addition $c$ has finite order then $\rho^c(\B_n)$ is finite.
 \end{enumerate}
\end{conj}

The validity of this conjecture would imply, for example, that if $(V,c)$ is a unitary BVS, $\{c\}$ is \emph{never} a universal gate set.

Unitary braid representations are abundant: for example Wenzl \cite{Wenzl} showed that the braided fusion categories associated with quantum groups at roots of unity are usually unitary, and hence produce many examples.
However, despite their ubiquity throughout mathematics and physics, there are remarkably few unitary BVSs in the literature.  A classification of BVSs is only known for $\dim(V)=2$ \cite{Dye,Hiet1}, for which Conjecture \ref{mainconj} follows from \cite{FRW,jenn}.  In this article we
focus on two infinite families: the so-called \emph{Gaussian} BVSs (see Section \ref{gaussian} and \cite{GJ}) and \emph{group-type} BVSs (see Section \ref{group-type}).  In particular, we verify Conjecture \ref{mainconj}(a) for group-type BVSs (see Corollary \ref{qBVS v.a.} and Conjecture \ref{mainconj}(b) for Gaussian BVSs (see Proposition \ref{gaussian finite}).
We give a generalization of BVSs of group-type to quasi-BVSs of group-type and prove our results in this more general setting.

The Gaussian braid representations here are (generalizations of) those considered in \cite{HNWphys}.  In \emph{loc. cit.} these appear as the braiding for anyonic models generalizing Ising anyons and Majorana
 zero modes.  The associated TQFT is the $SO(N)_2$ Chern-Simons theory, a fact that is established in \cite{RWe}.
Similarly, the braid representations associated with group-type BVSs considered here are generalizations of those obtained from Dijkgraaf-Witten TQFTs (\cite{DW}).

\section{Preliminary Results}\label{Preliminary Results}

Conjecture \ref{mainconj} should be compared with \cite[Conjecture 2.7]{RW}, where the predicted conclusion is that $\rho^c(\B_n)$ is finite modulo its center, assuming that $c$ has finite order.  Our minor strengthening is due to the following:
\begin{lemma}\label{modZ}
If $c$ has finite order and $\rho^c(\B_n)$ is finite modulo its center $Z^c$ then $\rho^c(\B_n)$ is itself
finite.
\end{lemma}
\begin{proof}
We decompose $\rho^c$
into irreducible representations $\rho_i$ and denote by $Z_i= \C I_i$ the
center of $\rho_i(\B_n)$.  As $c$ has finite order, the images of the
generators $\sigma_j$ in $\rho_i(\B_n)$ also have finite order.  Therefore,
every element of $\rho_i(\B_n)$ has determinant a root of unity (of finite
order).  Now suppose $\beta\in Z^c$ with decomposition $\beta=\sum_i
\lambda_iI_i$ as an element of $Z^c=\sum_i Z_i$.  Now we see that the
determinant of $\lambda_iI_i$ is a root of unity, and hence so is $\lambda_i$.
Therefore $\beta$ has finite order.  Thus the finitely generated abelian
group $Z^c$ is finite, as each element has finite order.  Thus $Z^c$ is finite and the result follows.
\end{proof}

 One large class of examples that obviously satisfy Conjecture \ref{mainconj}(a) are the so-called \emph{locally monomial BVSs} obtained as follows: Let $(S,X)$ be a set-theoretical solution (see \cite{ESS}) to the braid equation, \emph{i.e.}, $S:X\times X\rightarrow X\times X$ satisfies the Yang-Baxter equation in
$\Aut(X^{\times 3})$.  Then $S$ gives rise to a BVS $(V,S)$ via linearization.  Now choose a unitary diagonal matrix $D$ such that $(V,DS)$ is a BVS.  Observe that with respect to the basis $X$, $DS$ is a monomial matrix so that $\rho^{DS}(\B_n)$ is isomorphic to a subgroup of the group $M_{|X|^n}(\mathbb{T})$ of $|X|^n\times |X|^n$ monomial matrices with entries on the unit circle $\mathbb{T}$.  Now $M_{|X|^n}(\mathbb{T})$ is isomorphic to $\mathfrak{S}_{|X|^n}\rtimes \mathbb{T}^{|X|^n}$ (a wreath product) which has the abelian group $\mathbb{T}^{|X|^n}$ as a finite index subgroup.

This example motivates the following:

\begin{defn}

\begin{itemize}
\item A \textbf{locally monomial BVS} consists of the data $(V,c,X, $ $(V_x)_{x\in X}),$ where $(V,c)$ is a BVS, $X$ is a finite set, and $(V_x)_{x\in X})$ is a family of one dimensional subspaces of $V$ indexed by $X$, such that $V=\bigoplus_{x\in X}V_x$ and $c$ permutes the spaces $V_x\otimes V_{y}$ for all $x, y\in X$.

\item A BVS $(V,c)$ is called \textbf{locally monomializable} if there exist a finite set $X$ and a decomposition $(V_x)_{x\in X}$, such that $(V,c,X, (V_x)_{x\in X})$ is a locally monomial BVS.

\item Two locally monomial BVSs $(V,c_V,X, (V_x)_{x\in X})$ and $(W,c_W,Y, (W_y)_{y\in Y})$ are \textbf{monomial equivalent} if there is an isomorphism $f:V\to W$ of BVSs such that $f(V_x)= W_y$ for all $x\in X$.

\end{itemize}
\end{defn}

\begin{remark}
\begin{itemize}
\item[1.] The braiding of a locally monomial BVS induces a set-theo\-retical solution of the Yang-Baxter equation over $X$.

\item[2.] A BVS can have more than one inequivalent locally monomial structure.

\item[3.] If $(X,S)$ is a set-theoretical solution of the Yang-Baxter equation all possible locally monomial structures (up to diagonal isomorphism) are in one-to-one correspondence with the second Yang-Baxter cohomology group $H^2_{YB}(X,\mathbb C^*)$, see \cite{Carter}. 

\item[4.] Let  $(V,c_V,X, (V_x)_{x\in X})$ be a locally monomial BVS with $(X,S)$ the associated set-theoretical solution of the Yang-Baxter equation. If  there is a graded basis $\{v_x\in V_x\}_{x\in X}$, such that the scalars of $c(v_x\otimes v_y)= q_{x,y}v_{s_1(x,y)}\otimes v_{s_2(x,y)}$ are roots of the unity, then the image of the representation of $\B_n$ is a finite group.

\end{itemize}
\end{remark}

\section{Gaussian Braided Vector Spaces}\label{gaussian}
In this section we construct unitary, finite order BVSs of dimension $m$ generalizing the ``Gaussian'' representations
found in \cite{GJ,jonespjm,Jcmp89}, in which the $m$ odd cases are found.  We give full details as the case $m$ even is new, see also Remark \ref{typos}.

Let $m\in\N$ and define $q=\begin{cases}e^{2\pi i/m}, & m \;\mathrm{odd} \\ e^{\pi i/m}, &
m \;\mathrm{even.}\end{cases}$  Define, as in \cite{Jcmp89},
$ES(m,n-1)$ to be the algebra generated by
$u_1,\ldots,u_{n-1}$ subject to: $u_i^m=1$, $u_iu_{i+1}=q^2u_{i+1}u_i$ and
$u_iu_j=u_ju_i$ for $|i-j|>1.$
\begin{prop}\label{Gaussprop}

The mapping $\varphi_n(\sigma_i)=\frac{1}{\sqrt{m}}\sum_{j=0}^{m-1}q^{j^2}u_i^j$
defines a
representation of $\B_n$ into $ES(m,n-1)$ which becomes a $\ast$-representation
upon setting $u_i^*=u_i^{-1}$.
\end{prop}

We will need the following result, found in \cite[Example
2.18 and Section 3.3]{jonespjm}:
\begin{lemma}
 Let $G$ be a cyclic quotient of $\Z$ (i.e., written additively) and
$f:G\rightarrow\C^\times$ a function satisfying:

\begin{enumerate}
\item[(a)] $f(g)=f(-g)$
\item[(b)] $\frac{1}{|G|}\sum_{h\in G}\frac{f(h)}{f(g-h)}=\delta_{g,0}$
\item[(c)] $\sqrt{|G|}\frac{f(x+y)}{f(x)f(y)}=\sum_{g\in
G}\frac{f(g-y)f(g+x)}{f(g)}.$
\end{enumerate}

Then $\varphi^f:\sigma_i\rightarrow \sum_{g\in G}f(g)u_i^g$ gives a
homomorphism $\B_n\rightarrow ES(m,n-1)$.

\end{lemma}
We now proceed to:

\begin{proof}[Proof of Proposition \ref{Gaussprop}]
Let $m$ be odd and $G=\Z/m\Z$.  Define
$f(a)=Kq^{a^2}$ where $q=e^{2\pi i/m}$ as above and $K$ is a constant to
be determined.  Clearly $(a)$ is satisfied while $(b)$ is easily verified:

$$\sum_{h\in
G}\frac{f(h)}{f(g-h)}=\sum_{h=0}^{m-1}q^{h^2-(g-h)^2}=q^{g^2}\sum_he^{ 4gh\pi
i/m}=m\delta_{g,0}.$$

To verify $(c)$ we must show:

$$\sqrt{m}\frac{f(x+y)}{f(x)f(y)}=\sqrt{m}q^{2xy}/K=\sum_{g=0}^{m-1}\frac{f(g-y)f(g+x)}{f(g)}=q^{x^2+y^2}K\sum_{g=0}^{m-1}q^{
g^2+2g(x-y)} .$$
Completing the square we see that it suffices to find a constant $K$ satisfying:
$$\sqrt{m}/K^2=\sum_{g=0}^{m-1}q^{(g+(x-y))^2}=\sum_{g=0}^{m-1}q^{g^2}$$ since $\{(g+(x-y)):0\leq g\leq
m-1\}$ is a complete set of residues modulo $m$ (for any fixed $x,y$).  By Gauss' famous result we have
$\sum_{g=0}^{m-1}q^{g^2}=\begin{cases} \sqrt{m} &
m=1\pmod{4}\\
                                      i\sqrt{m} & m=3\pmod{4}.
                                     \end{cases}$
  Thus, $K^2=1$ or $K^2=-i$ will give solutions in
these two cases.
In order to get a
$\ast$-representation of $\B_n$ in $ES(m,n-1)$ one should rescale by
$1/\sqrt{m}$.

Now let $m$ be even and $G=\Z/m\Z$.  Recall that here $q=e^{\pi i/m}$.
Define $f(a)=Kq^{a^2}$ as above.  Conditions $(a)$ and $(b)$ are verified in the same way
as in the odd case, noting that $\sum_he^{2\pi g i/m}=m\delta_{g,0}$.
Condition $(c)$ reduces to verifying that
$\sqrt{m}/K^2=\sum_{g=0}^{m-1}q^{(g+(x-y))^2}$ has a constant solution $K$.
Since $q$ is a $2m$th root of
unity, this is not a standard quadratic Gauss sum.  However,
 $(g-k)^2=g^2\pmod{2m}$ so that
$$\sum_{g=0}^{m-1}e^{\frac{2\pi
i(g-k)^2}{2m}}=\frac{1}{2}\sum_{g=0}^{2m-1}e^{\frac{2\pi i
g^2}{2m}}=\frac{(i+1)\sqrt{2m}}{2}$$ (see \cite[Chapter 2]{Davenport}). Thus either choice for $K^{-2}=e^{\pi i/4}=\frac{1+i}{\sqrt{2}}$  gives
a solution.\end{proof}

\begin{remark}\label{typos}\begin{itemize}
                      \item  If we were to use $q=e^{\frac{2\pi i}{m}}$ for $m$ even the operator $\sum_{j=0}^{m-1}q^{j^2}u_i^j$ would
not be invertible.
\item In \cite{jonespjm} the formula for $f$ given as $Ke^{\frac{\pi i a^2}{m}}$
with $K^{-2}=\sum_{g=0}^{m-1} e^{\frac{2\pi g^2i}{m}}$ has typos, which are corrected
in \cite{GJ,Jcmp89}.

                           \end{itemize}
                           \end{remark}

\subsection{Gaussian BVSs and Braid Group Images}

To produce a Gaussian BVS it is enough to exhibit a vector space $V$ and an operator $U\in\Aut(V^{\ot 2})$ so that $$u_i\rightarrow U_i:=I_V^{\ot i-1}\ot U\ot I_V^{\ot
n-i-1}$$ extends to an algebra homomorphism $ES(m,n-1)\rightarrow \End(V^{\ot n})$, \emph{i.e.} a localization of $ES(m,n-1)$.
Recalling that $q=e^{2\pi i/m}$ for $m$ odd and $q=e^{\pi
i/m}$ for $m$ even, $\{\e_i:0\leq i\leq m-1\}$ be the standard basis for
$V=\C^m$ and define a $U\in\End(V^{\ot 2})$ by $$U(\e_i\ot
\e_j)=q^{j-i}\e_{i+1}\ot\e_{j+1},$$ where $\e_{i+m}:=\e_i$.
\begin{prop}
 The map $u_i\rightarrow U_i:=I^{\ot i-1}\ot U\ot I^{\ot
n-i-1}\in\End(V^{\ot n})$ defines a $\ast$-algebra homomorphism $ES(m,n-1)\rightarrow \End(V^{\ot n})$.  Moreover, $R:=\frac{1}{\sqrt{m}}\sum_{j=0}^{m-1} q^{j^2}U^j$ is a unitary operator.
\end{prop}
\begin{proof}

It is immediate that
$U_i$ commutes with $U_j$ if $|i-j|>1$.  Moreover,

$$U^m(\e_i\ot
\e_j)=q^{(j-i)(m-j)}q^{(m-j+i)(j-i)}q^{(j-i)i}
(\e_i\ot\e_j)=(\e_i\ot \e_j).$$
To check the remaining relation in
$ES(m,n-1)$ it suffices to check the case $n=3$ with $U_1$ and
$U_2$:

\begin{eqnarray*}
 U_1U_2(\e_i\ot\e_j\ot\e_k)=q^{1+k-i}\e_{i+1}\ot\e_{j+2}\ot
\e_{k+1}=\\
q^2q^{k-i-1}\e_{i+1}\ot\e_{j+2}\ot\e_{k+1}=q^2U_2U_1(\e_i\ot\e_j\ot\e_k).
\end{eqnarray*}
Since $U^{-1}=U^*$ (conjugate-transpose) the homomorphism respects the $\ast$-structure on $ES(m,n-1)$ and hence $R$ is unitary.

\end{proof}
We will call the pair $(V,R)$ above a \emph{Gaussian} BVS.

Let $R_i(k):=\frac{1}{\sqrt{k}}\sum_{j=0}^{k-1}q^{j^2}u_i$ be the braid
operators
in $ES(k,n-1)$ where $q$ is a $k$-th (or $2k$-th) root of unity for $k$ odd
(resp. $k$ even).
\begin{lemma}\label{primelemma}
 Suppose that we have prime factorization $m=p_1^{a_1}p_2^{a_2}\cdots
p_s^{a_s}$. Then there exists Galois automorphisms
$\tau_h\in Gal(\Q_{p_h^{a_h}})$ so that
$$R_i(m)=\prod_{h=1}^{s}\tau_h(R_i(p_h^{a_h})).$$
In particular, if $G_h$ is the image of $\B_n$ in $ES(p_h^{a_h},n-1)$ then the
image of $R_i(m)$ is a (diagonal) subgroup of $\prod_{h=1}^{s}\tau_h(G_h)$.

\end{lemma}
\begin{proof}
 The key observation here is the following: if $x$ and $y$ are coprime then
$u_i^x$ and $u_j^y$ commute for all $i$ and $j$, and $u_i^x$ generates an
algebra of the form $\tau(ES(y,n-1))$ and vice versa.  So in particular
$ES(xy,n-1)$ factors as a direct product of algebras isomorphic to $ES(x,n-1)$
and $ES(y,n-1)$.  The Chinese Remainder Theorem permits one to find
automorphisms $\gamma_x$ and $\gamma_y$ so that
$R_i(xy)=\gamma_x(R_i(x))\gamma_y(R_i(y))$.
\end{proof}

We may now verify Conjecture \ref{mainconj}(b) for the Gaussian BVS $(V,R)$.  Setting $R_i=\rho^R(\sigma_i)$ as usual we have:
\begin{prop}\label{gaussian finite}
 The group $G_n$ generated by $R_1,\ldots,R_{n-1}$ is a finite group.
\end{prop}

\begin{proof}
Firstly, since $G_{n-1}\subset G_{n}$, it is enough to consider $n$ odd.
By an easy calculation we find
\begin{eqnarray}
 R_iu_{i+1}R_i^{-1}&=&qu_i^{-1}u_{i+1}\\
R_iu_{i-1}R_i^{-1}&=&q^{-1}u_{i-1}u_i.
\end{eqnarray}
Therefore, the conjugation action of $G_n$ on the finite set
$T_n:=\{q^{a_0}u_1^{a_1}\cdots u_{n-1}^{a_{n-1}}: 0\leq a_i\leq 2m-1\}$ gives a homomorphism $\Psi_n:G_n\rightarrow Sym(T_n)$.  The kernel of $\Psi_n$ is a subgroup of the center of $ES(m,n-1)$
which consists of scalar multiples of $1$ (for $n$ odd).  Since the eigenvalues of $R_i$
are roots of unity (of finite order) the determinant of any element of
$\ker(\Psi_n)$ is also a finite order root of unity.  Since it is a scalar
matrix it must have finite order.
\end{proof}

\section{Quasi-braided vector spaces of group-type}\label{group-type}
The following definition appears in \cite{AS}.

\begin{defn}
 A braided vector space $(V,c)$ is said to be of \textbf{group-type} if
there exists a basis $\{x_1,\ldots,x_n\}$ for $V$ and $g_i\in GL(V)$ such that
 $c(x_i\ot z)=g_i(z)\ot x_i$ for all $z\in V$.  We
will call such a basis a \textbf{braided basis} for $c$.
\end{defn}

In this section we will give an equivalent definition of group-type BVSs which can be extended to group-type quasi-BVSs.  We will then verify Conjecture \ref{mainconj}(a) for these classes of BVSs.
\subsection{Twisted Yetter-Drinfeld modules over a group}

Let $G$ be a group and $\omega\in Z^3(G,U(1))$ a 3-cocycle. Let us define
$\gamma:G\times G\to \text{Map}(G,U(1))$ by
$$\gamma_{\sigma,\tau}(\rho):=\frac{\omega(\sigma,\tau,
\rho)\omega(\sigma\tau\rho(\sigma\tau)^{-1},\sigma,\tau)}{\omega(\sigma,
\tau\rho\tau^{-1},\tau)}$$
and $\mu:G\to \text{Map}(G\times G,U(1))$ by
$$\mu_{\sigma}(\tau,\rho):= \frac{\omega(\sigma\tau\sigma^{-1},\sigma,\rho)}{\omega(\sigma\tau \sigma^{-1},
\sigma\rho\sigma^{-1},   \sigma)\omega(\sigma,\tau,\rho)}$$ for all $\sigma,\tau,\rho \in G$.

We denote by $\YD_G^\omega$ the category of Yetter-Drinfeld  modules twisted by
$\omega$, defined as follows:  an object of $\YD_G^\omega$ is a vector space with a
decomposition $V =\bigoplus_{\sigma\in G}V _\sigma$ and a compatible
$\gamma$-twisted $G$-action, \textit{i.e.},  a map $G\times V\to V$,  such that
 $\sigma V_\tau\subset V_{\sigma \tau \sigma^{-1}}$  and $$(\sigma\tau) v_\rho=
\gamma_{\sigma,\tau}(\rho)\sigma (\tau v_\rho)$$  for all $v_\rho \in V_\rho$ and all
$\sigma,\tau,\rho \in G$.

A morphism in $\YD_G^\omega$ is a linear map  $f:V\to W$ such that
$f(V_\sigma)\subset W_{\sigma}$ and $f(\sigma v)= \sigma f(v)$ for all $\sigma
\in G, v\in V$.

The category of $\omega$-twisted Yetter-Drinfeld modules is
braided with: the tensor product in $\YD_G^\omega$  is  the tensor
product of vector spaces with  $G$-grading $$(V\otimes W)_\sigma=
\bigoplus_{\tau \in G} V_{\sigma\tau^{-1}}\otimes W_{\tau},$$  and
twisted $G$-action $\sigma v_\tau\otimes w_\rho:=
\mu_{\sigma}(\tau,\rho)\sigma v_\tau\otimes \sigma w_\rho$ for all
$v_\sigma \in V_\sigma, w_\tau \in W_\tau$, $\sigma,\tau, \rho \in
G$. The associator is given by
\begin{align}
a_{V,W,Z} :(V\otimes W)\otimes Z &\to V\otimes (W\otimes Z)\\
(v_\sigma\otimes w_\tau)\otimes z_\rho &\mapsto
\omega(\sigma,\tau,\rho)v_\sigma\otimes (w_\rho\otimes z_\rho)
\end{align} and the braiding by $c_{V,W}(v_\sigma\otimes w_\tau)=\sigma
w_\tau\otimes v_\sigma$.

When $\omega$ is trivial, $\YD_G^\omega$ is just the usual
category of Yetter-Drinfeld modules over a group algebra (see
\cite{AG}) and if $G$ is finite,  $\YD_G^\omega$ is braided monoidally equivalent
to the (modular) category $\Rep(D^\omega G)$ of representations of the
twisted Drinfeld double $D^\omega G$, defined by Dijkgraaf,
Pasquier and Roche \cite{DPR}.

The \textbf{support} of $V\in\YD_G^\omega$
is the set $$\supp(V)=\{g\in G:
\dim(V_g)\neq 0\}.$$
The following result show the connection between Yetter-Drinfeld-modules and
BVSs of group-type.

\begin{prop}\label{prop relation BVS GT and YD-mod}
A BVS of group-type $(V,c)$ is the same as a finite-dimensional
Yetter-Drinfeld module $V$ over a group $G$ such that the associated
$G$-action is faithful and $G$ is generated by the support of $V$.
\end{prop}
\begin{proof}
Let $(V,c)$  be a BVS of group-type with braided basis $\{x_1,\ldots,x_n\}$ and corresponding $\{g_1,\ldots,g_n\}\in GL(V)$. Define
$G:=\langle g_1,\ldots, g_n\rangle \subset \text{GL}(V)$ and
$V_{g_i}:=\text{span}\{x_k| g_k=g_i\}$. The braid equation on $c(x_i\otimes
x_j)=g_i(x_j)\otimes x_i$ is equivalent to $g_i V_{g_j}\subset
V_{g_ig_jg_i^{-1}}$. The converse is obvious.
\end{proof}
For the rest of this section $(V,c)$ will denote a BVS of group-type, with
$c(x_i\ot x_j)=g_i(x_j)\ot x_i$ for some ordered braided basis
$[x_1,\ldots,x_n]$ for
$V$ and $g_i\in GL(V)$.  For notational convenience let $G=\lan
g_1,\ldots,g_n\ra$ be the subgroup of $GL(V)$ generated by the $g_i$.

\begin{remark}
We should emphasize that
BVSs of group type are precisely all possible BVSs that
can be obtained as \emph{finite dimensional} Yetter-Drinfeld
modules over an \emph{arbitrary} group $G$.
\end{remark}

In \cite[Definition 3.6]{GHR}, we defined the notion of \emph{quasi-braided vector
spaces} $(V,a,c)$ incorporating the possibility of non-trivial associativities through a family of isomorphisms $a$.  Following the
relation between BVSs of group-type and Yetter-Drinfeld modules
(Proposition \ref{prop relation BVS GT and YD-mod}), we propose
the  definition of  a quasi-BVS of group-type as follows:

\begin{definition}\label{qBVS of gt}
A \textbf{quasi-BVS of group-type} is a finite-dimensional
twisted Yetter-Drinfeld module over a group $G$ such that the associated
twisted $G$-action is faithful and $G$ is generated by the support
of $V$.
\end{definition}

If $G$ is a finite group, the quasi-BVS of group type associated with a twisted $\YD$-module $V$ is the quasi-BVS of \cite[Example 3.7]{GHR}.  In this case, we have the following:

\begin{prop}\label{Gfinite}
 If $(V,c)$ is a quasi-BVS of group type and
$|G|<\infty$ then $|\rho^c(\B_n)|$ is a finite group for all $n$.
\end{prop}
\begin{proof}
 One may identity $V$ with an object in $\Rep(D^\omega G)$ such that the braiding on $V^{\ot
2}$ obtained from the universal $R$-matrix of $DG$ coincides
with $c$ (see e.g. \cite{AS}).  Now by \cite[Theorem 4.2]{ERW} the
corresponding braid group representation has finite image.
\end{proof}

\subsection{Supporting subgroup structure}\label{description central extensions}
In light of Proposition \ref{Gfinite}, from here on \emph{we will assume that $G$ is an infinite group}.

Let $V$ be a finite dimensional twisted $\YD$-module over an
arbitrary group $G$. Since $V$ is finite dimensional supp$(V)$ is
a finite set, and if $g\in\supp(V)$ the conjugacy class of $g$ is
finite (since $\sigma V_g\subset V_{\sigma g \sigma^{-1}}$).  If
$F$ is the subgroup of $G$ generated by supp$(V)$, then $V$ is a
twisted $\YD$-module over $F$ and the representation of the braid
group are exactly the same as the representation associated to the
quasi-BVS  $V$ as twisted $\YD$-module over $G$.

Therefore, as we are interested in the study of
the quasi-BVSs associated with twisted $\YD^\omega_G$-modules,
\emph{we can and will assume that $G$ is generated by a finite
number of elements with finite conjugacy classes.}

\begin{lemma}
Let $G$ be a group generated by a finite number elements each conjugacy class of which is finite.
Then the center $\Zz(G)$ and centralizers $C_G(\sigma)$ have finite index for any $\sigma\in G$.
\end{lemma}
\begin{proof}
Let $\{g_1,\ldots,g_n\}$ be a set of generators such that each $g_i$ has a
finite conjugacy class. Clearly $$\Zz(G)=\bigcap_{i=1}^n C_G(g_i),$$
\textit{i.e.}, an element of $G$ lies in the center if and only if it commutes
with all the generators. Since for each $g_i$ we have $[G:C_G(g_i)]<\infty$ and
$\Zz(G)$ is an intersection of a finite number of subgroups of finite index,
then $\Zz(G)$ has also finite index.

Now, since $\Zz(G)$ has finite index and
$\Zz(G) = \bigcap_{\sigma \in G} C_G(\sigma)$, $C_G(\sigma)$ has finite index for all $\sigma \in G$.
\end{proof}

\begin{cor}
If $(V,c)$ is a quasi-BVS of group-type with corresponding group $G:=\langle g_1,\ldots, g_n\rangle$, then $G/\mathcal{Z}(G)$ is a finite group.
\end{cor}


Since $G$ is finitely generated then $\Zz(G)$ is also finitely generated. Let
$A$ the torsion-free subgroup of $\Zz(G)$ and $H:= G/A$. Since $G/\Zz(G)$
and $\Zz(G)/A$ are finite groups, $H$ is also finite.


The group $G$ is a central extension of $H$ by $A$. The central extensions of
$H$ by $A$ are classified by elements of $H^2
(H, A)$. If $A$ has rank $r$, \emph{i.e.}, $A\cong \oplus_{i=1}^r\mathbb{Z}$ then $H^2(H,A)\cong
\oplus_{i=1}^rH^2(H,\mathbb{Z})$ and $H^2(H,\mathbb{Z})\cong
\text{Hom}(H,U(1))$. Then if $H/H'\cong \oplus_{i=1}^m \mathbb{Z}_{n_i}$,
$$H^2(H,A)\cong \oplus_{i=1}^m A/n_iA.$$
The 2-cocycles (hence the central extensions) can be constructed
explicitly as follows: choose a 2-cocycle $f\in
Z^2(U(1),\mathbb{Z})$ that represents the exact sequence
$$\mathbb{Z}\to \mathbb{R}\to U(1)$$ and define for $\chi\in
\text{Hom}(G,U(1)^{\times r})$, the 2-cocycle $f_\chi:H\times H\to
A:=\oplus_{i=1}^r \mathbb{Z}$, where $$f_\chi(h,h')=f^{\times
r}(\chi(h),\chi(h'))$$ for all $h,h'\in H$. Note that by
construction $f_\chi$ is symmetric. In fact, since $A$ is
torsion-free for every finite group, $H^2(G,A)=Ext(G/G',A)$.

\subsection{Irreducible unitary twisted $\YD$-modules}

In this subsection we will classify the finite dimensional irreducible twisted
Yetter-Drinfeld modules.
\subsubsection{Finite dimensional unitary projective representations of central extensions by finite groups}

Let $G$ be a  group and $\alpha \in Z^2(G,U(1))$ a 2-cocycle. Since every 2-cocycle is cohomologous to a standard 2-cocycle, we will assume that $\alpha$ is standard, i.e.,  $\alpha(\sigma,\sigma^{-1})=1$ for all $\sigma \in G$.

We define subgroups
\begin{align*}
\Zz(G)(\alpha) &:=\{\sigma \in \Zz(G):
\alpha(\sigma,\tau)=\alpha(\tau,\sigma), \text{for all } \tau \in \Zz(G)\},\\
G(\alpha) &:=\{\sigma \in G: \alpha(\sigma,\tau)=\alpha(\tau,\sigma), \text{for all } \tau \in \Zz(G)(\alpha) \}.
\end{align*}
It is not difficult to see that the groups $\Zz(G)$ and $G(\alpha)$ only depend
on the cohomology class of $\alpha$.

\begin{example}
Let $q$ be a primitive $n$-th root of unity. We can define a 2-cocycle over $A=\mathbb{Z}\oplus \mathbb{Z}$, by $\alpha(a_1\oplus a_2,b_1\oplus b_2)= q^{a_1b_2}$, then $A(\alpha)=n\mathbb{Z}\oplus n\mathbb{Z}$.
\end{example}
Given a discrete group $G$ and a 2-cocycle $\alpha\in Z^2(G,\mathbb{C}^*)$, we will denote by $\mathbb{C}_\alpha[G]$ the  group algebra twisted by $\alpha$, that is, $\mathbb{C}_\alpha[G]$ is the vector space with basis $\{u_\sigma\}_{\sigma\in G}$ and product $u_\sigma u_\tau=\alpha(\sigma,\tau)u_{\sigma\tau}$, for all $\sigma, \tau \in G$. The category of left $\mathbb{C}_\alpha[G]$-modules is canonically  isomorphic to the category of $\alpha$-projective representation of $G$ and will be denoted by $\Rep(\mathbb{C}_\alpha[G])$.

Let $G$ be any group, $N\trianglelefteq G$ be a normal subgroup and $\pi:
G\times G\to G/N\times G/N$ the canonical projection. For any $\omega\in
Z^2(G/N,U(1))$, Inf$(\omega):=\omega\circ \pi \in Z^2(G,U(1))$. The map Inf is
called the inflation map and it defines a group homomorphism Inf$:
H^2(G/N,U(1))\to H^2(G,U(1))$. Note that the natural projection
$\Pi:\C_{\Inf(\omega)}[G]\to \C_{\omega}[G/N]$ is an algebra epimorphism, so
it defines a faithful functor $$\Inf(-):=\Pi^*(-):
\text{Rep}(\C_{\omega}[G/N])\to \text{Rep}(\C_{\Inf(\omega)}[G]).$$

The following theorem is a generalization of the main results of \cite{B}.
\begin{theorem}\label{irreducible projective repr}
Let $G$ be a finitely generated group such that $G/\Zz(G)$ is finite and $\alpha
\in Z^2(G,U(1))$ be a 2-cocycle. Then

\begin{enumerate}
\item $G$ has a finite dimensional
irreducible $\alpha$-projective $G$-representation if and only if $\Zz(G)(\alpha)$ has finite index.

\item If $\Zz(G)(\alpha)$ has finite index, then there is a  2-cocycle
$$\Omega\in Z^2(G(\alpha)/\Zz(G)(\alpha),U(1))$$ such that $\alpha|_{G(\alpha)}$
is cohomologous to Inf$(\Omega)$.

\item Assume that $\alpha|_{G(\alpha)}=$Inf$(\Omega)$ and that
$\{U_1,\ldots,U_l\}$ is a representative set of the irreducible
$\Omega$-representations of $G(\alpha)/\Zz(G)(\alpha)$ (since
$G(\alpha)/\Zz(G)(\alpha)$ is finite, there are only finitely many isomorphism
classes). Then all irreducible finite dimensional representations of $G$ are of
the form $$Ind_{G(\alpha)}^G(\theta\otimes \Inf(U_i)),$$ where
$\theta:G(\alpha)\to U(1)$ is a character.

\item A pair of  representations Ind$_{G(\alpha)}^G(\theta\otimes \Inf(U_i))$ and Ind$_{G(\alpha)}^G(\theta'\otimes \Inf(U_j))$ are isomorphic if and only if $i=j$ and $\theta|_{\Zz(G)(\alpha)}=\theta'_{\Zz(G)(\alpha)}$.

\end{enumerate}
\end{theorem}
\begin{proof}
Since $G$ is finitely generated, $\Zz(G)$ is finitely generated, so every finite
dimensional unitary projective representation of $\Zz(G)$ is the  direct sum of
its irreducible subrepresentations.
Let $V$ be a finite dimensional irreducible $\alpha$-represen\-tation of $G$. Then $Res^G_{\Zz(G)}(V)$ is a finite dimensional unitary $\alpha$-representation of $\Zz(G)$, so by \cite[Lemma 2]{B},
$[\Zz(G):\Zz(G)(\alpha)]$ is finite and since $[G:\Zz(G)]$ is finite,  $[G:\Zz(G)(\alpha)]$ is finite.

We will apply Mackey's theory \cite{Mackey} to the finite index central subgroup $\Zz(G)(\alpha)$ of $G$.
Since $\alpha$ is symmetric on $\Zz(G)(\alpha)$, its cohomological class is trivial and we may
replace $\alpha$ by a 2-cocycle such that  $\alpha|_{\Zz(G)(\alpha)}\equiv 1$.
Let $V$ be a unitary finite dimensional $\alpha$-projective repre\-sentation of $G$. Since $\alpha|_{\Zz(G)(\alpha)}\equiv 1$, the restriction of $V$ to   $\Zz(G)(\alpha)$ decomposes as a direct sum of one dimensional linear representation of $\Zz(G)$. It is easy to check that the little $\alpha$-group (i.e., stability group) of $G$ of any of these one dimensional $\Zz(G)$ representations is equal to $G(\alpha)$ and only depends on $\alpha$.  In particular $G(\alpha)$ is normal. By \cite[Theorem 8.2]{Mackey} it follows that there exists an $\alpha$-projective irreducible unitary representation $W$ such that $V=\text{Ind}_{G(\alpha)}^G(W)$.  Since $G(\alpha)$ is normal,  $V=\text{Ind}_{G(\alpha)}^G(W)$ is isomorphic to $V'=\text{Ind}_{G(\alpha)}^G(W')$ if and only if $W$ is isomorphic to $W'$ as $\alpha$-projective $G(\alpha)$-representations. Then the induction functor  $\text{Ind}_{G(\alpha)}^G(-)$ defines an equivalence between the category of finite dimensional unitary $\alpha$-projective re\-presentation
of
$G(\alpha)$ and the category of finite dimensional unitary $\alpha$-projective of
$G$.

As in the proof of the main result of \cite{B}, we will apply  Mackey's theory again, this time to the central subgroup $\Zz(G)(\alpha)$ of $G(\alpha)$. Let $\chi:\Zz(G)(\alpha)\to U(1)$ be a one dimensional representation of $\Zz(G)$ and recall that the little $\alpha$-group of $G(\alpha)$ at $\chi$ is $G(\alpha)$ itself. By \cite[Theorem 8.2]{Mackey} there exists a 2-cocycle $\Omega\in Z^2(G(\alpha)/\Zz(G)(\alpha))$ and a 1-cochain $\theta \in C^1(G(\alpha),U(1))$ such that
\begin{equation}\label{extension}
\delta(\theta)=\alpha\Inf(\Omega)^{-1}
\end{equation}
Then the cohomology class of $\alpha$ and Inf$(\Omega)$ are the same, so we can assume that $\alpha=$Inf$(\Omega)$. Now, since $\alpha=$Inf$(\Omega)$, equation \eqref{extension} implies that $\theta$ is a character, so again by  \cite{Mackey} (see also \cite[Theorem 6.4.2]{karpi}) the theorem follows.
\end{proof}

\begin{cor}
Let $G$ be a finitely generated group such that $G/\Zz(G)$ is finite. Then

\begin{enumerate}

\item There is a  2-cocycle $\Omega\in Z^2(G/\Zz(G),U(1))$ such that
Inf$(\Omega)$ has trivial cohomology  and if
$\{U_1,\ldots,U_l\}$ is a representative set of the irreducible
$\Omega$-representations of $G/\Zz(G)$, then all irreducible finite dimensional
representations of $G$ are of the form $$\theta\otimes \Inf(U_i),$$ where
$\theta:G\to U(1)$ is a character.

\item A pair of  representations  $\theta\otimes Inf(U_i)$ and $\theta'\otimes Inf(U_j)$ are isomorphic if and only if $i=j$ and $\theta|_{\Zz(G)}=\theta'_{\Zz(G)}$.

\end{enumerate}

\end{cor}
\qed


\begin{theorem}\label{cor M(chi)}
Let $G$ be a finitely generated group such that $G/\Zz(G)$ is finite and $\alpha
\in Z^2(G,U(1))$ a 2-cocycle. Then every finite dimensional unitary
$\alpha$-projective $G$-representation is a subrepresentation of a finite
dimensional monomial $\alpha$-projective $G$-representation.
\end{theorem}
\begin{proof}
Since every  finite dimensional unitary $\alpha$-projective
$G$-representation is completely reducible, it is enough to proof
the corollary for a finite dimensional irreducible unitary
representation. By Theorem \ref{irreducible projective repr}(1)(2),
we have that $G(\alpha)/\Zz(G)(\alpha)$ is a
finite group  and we can assume that $\alpha|_{G(\alpha)}=
\Inf(\Omega)$ for some $$\Omega\in
Z^2(G(\alpha)/\Zz(G)(\alpha),U(1)).$$

By Theorem \ref{irreducible projective repr}(3), if  $V$ is
an irreducible finite dimensional unitary $\alpha$-representa\-tion, then
$V\cong \text{Ind}_{G(\alpha)}^G(\theta\otimes \Inf(U_i))$,
where $U_i$ is an irreducible $\Omega$-representation of
$G(\alpha)/\Zz(G)$. We can construct the finite dimensional
monomial $\alpha$-projective representation
$M(\theta):=\text{Ind}_{G(\alpha)}^G(\theta\otimes \Inf(\mathbb
C_{\Omega}[G(\alpha)/\Zz(G)(\alpha)]))$ and $V$
is a subrepresentation of $M(\theta)$.
\end{proof}

\subsubsection{Classification of Twisted $\YD$-modules}

\begin{defn}

\begin{itemize}
\item Let $G$ be a group, $V$ be a vector space and $\pi: G\to
\text{GL}(V)$ a map. A \textbf{monomial structure} on the map $\pi$
is a decomposition $V=\bigoplus_{x\in X}V_x$ into subspaces of
dimension one, such that $\pi(\sigma)$ permutes the $V_x$ for all
$x\in X$ and $\sigma \in G$.

\item A \textbf{monomial object} in $\YD_G^\omega$ consist of a twisted
$\YD$-module, with a monomial structure on the twisted
$G$-action.

\item A \textbf{monomial projective representation} is a projective
representation $\pi: G\to \text{GL}(V)$ with a monomial structure.

\end{itemize}
\end{defn}


Fix $\chi\in Hom(H; U(1)^{\times r})$ and let $G := A\rtimes_\chi
H$ the associated central extension (see Subsection
\ref{description central extensions}). For any subset $S\subset H$
we will denote  the subset $\{(a,s): a\in A, s\in S\}$ of $G$ by $A \rtimes_\chi S$.

\begin{lemma}
\begin{itemize}
\item For all $(a,\sigma)\in G$, $C_G((a,\sigma))=A\rtimes_\chi
C_H(\sigma)$, in particular $\Zz(G)=A\rtimes_\chi \Zz(H)$. \item If $T$
is a representative set of conjugacy classes of $H$, then
$A\rtimes_\chi T$ is a representative set of the conjugacy classes of
$G$ and for all $(a,\sigma)$,
$[G:C_G((a,\sigma))]=[H:C_H(\sigma)]$.
\end{itemize}
\end{lemma}
\begin{proof}
It is easy to see that $(a,\sigma)$ commutes with $(b,\tau)$ if
and only if $\sigma\tau=\tau\sigma$ and
$f_\chi(\sigma,\tau)=f_\chi(\tau,\sigma)$, but by construction
$f_\chi$ is symmetric so $(b,\tau)\in C_G((a,\sigma))$ if and
only if $\tau\in C_H(\sigma)$. The second statement is a consequence of the
first.
\end{proof}

For each conjugacy class $R$ we denote by $\YD_G^\omega(R)$ the full abelian subcategory of $\YD_G^\omega$ consisting of the zero object and all twisted $\YD$-modules with support $R$.
We will also denote by $(a,\sigma)^G$ the conjugacy class of $(a,\sigma)$ in $G$, that is, $(a,\sigma)^G=\{(a,\tau): \tau \in \sigma^H\}$.

\begin{prop}\label{prop desc YD}
Let $T$ be a representative set of conjugacy classes of $H$.
\begin{enumerate}
\item $\YD_G^\omega=\bigoplus_{(a,\sigma )\in A\rtimes_\chi T}\YD_G^\omega((a,\sigma)^G)$.

\item The abelian categories $\YD_G^\omega((a,\sigma)^G)$ and $\Rep(\mathbb C_{\gamma_{-,-}(\sigma)}[A\rtimes_\chi C_H(\sigma)])$ are equivalent.
\item The finite dimensional irreducible unitary twisted $\YD$-modules are in one-to-one correspondence
with triples $(\sigma,\theta,U)$, where $\sigma\in T$, $U$ is a finite dimensional unitary
irreducible projective representation of $$A\rtimes_\chi
C_H(\sigma)(\gamma_{-,-}(\sigma))/\Zz (A\rtimes_\chi
C_H(\sigma))(\gamma_{-,-}(\sigma))$$ and  $\theta:A\rtimes_\chi
C_H(\sigma)(\gamma_{-,-}(\sigma))\to U(1)$ is a character.
\end{enumerate}
\end{prop}
\begin{proof}
Statements (1) and (2) are simply restatements, for $G=A\rtimes_\chi H$, of the description in \cite{DPR} of irreducible representations of $D^\omega G$ for $G$ finite.  Statement (3) follows from Theorem \ref{irreducible projective repr}.
\end{proof}

\begin{theorem}\label{Theo exist monomial quotient}
Every finite dimensional unitary twisted $\YD$-module over a group $G$ is a
subobject and a quotient of a monomial twisted $\YD$-module over some group.
\end{theorem}
\begin{proof}
Since the category of finite dimensional unitary twisted $\YD$-modules over a group is semisimple, it is enough to prove that every unitary irreducible twisted $\YD$-module is a subobject and a quotient of a monomial twisted $\YD$-module.

If $U$ is a monomial $\gamma_{-,-}(\sigma)$-projective representation of $A\rtimes_\chi C_H(\sigma)$, it follows by definition  that the associated twisted $\YD$-module in Proposition \ref{prop desc YD} (see \cite{DPR} for details of the construction)  is monomial. If $V$ is a finite dimensional unitary irreducible twisted $\YD$-module, it is supported by only one conjugacy class. Let  $U$ be the irreducible $\gamma_{-,-}(\sigma)$-projective representation of $A\rtimes_\chi C_H(\sigma)$ associated to $V$, it follows by Theorem \ref{cor M(chi)} that there exists a monomial $\gamma_{-,-}(\sigma)$-projective representation of $A\rtimes_\chi C_H(\sigma)$ such that $U$ is a subobject and its associated twisted $\YD$-module is monomial and contains $V$.

\end{proof}

\begin{theorem}
The image of the braid group representations associated to
any finite dimensional unitary twisted $\YD$-module over a group $G$ is a quotient
of a monomial representation of the braid group, defined via a
monomial twisted $\YD$-module.
\end{theorem}
\begin{proof}
Let $V$ be a finite dimensional unitary twisted $\YD$-module. By Theorem \ref{Theo exist monomial quotient} there exists a monomial twisted $\YD$-module $W$ such that $V$ is a quotient (or subobject) of $W$. By the functoriality of the braiding in $\YD^\omega_G$, the image of the representation of $\B_n$ is a quotient of the representation associated to $W$. Since $W$ is monomial, the associated representation of $\B_n$ is monomial.
\end{proof}

We can now easily prove:
\begin{cor}\label{qBVS v.a.}
The braid group image associated with any unitary quasi-BVS of group-type is virtually abelian.
\end{cor}
\begin{proof}
A quasi-BVS of group-type is, by Definition \ref{qBVS of gt}, the same as a finite dimensional unitary twisted $\YD$-module over a group $G$ and the corresponding braid group images coincide.  Since the image of any monomial representation of the braid group is virtually abelian the result follows.
\end{proof}




\section{Discussion}

A classification of unitary BVSs in dimension higher than $2$ seems computationally well out-of-reach.  However, it is feasible that the monomial and Gaussian BVSs generate a large proportion of them (e.g. through quotients and subrepresentations).  We are not aware of any BVSs that do not come from these two families.

We think the following questions are worth pursuing:
\begin{enumerate}
 \item Does Conjecture \ref{mainconj}(a) imply (b)?  That is, if $(V,c)$ is a BVS such that $c$ has finite order and the braid group image is virtually abelian does it follow that the image is actually finite?
 \item Does there exist a locally monomial BVS $(V,c)$ such that the braid group image is not finite modulo the center?  By Lemma \ref{modZ} $c$ must have infinite order.  Moreover, there are monomial representations of $\B_n$ with infinite image modulo the center, the following gives an example:  $$\sigma_1\rightarrow\left[ \begin {array}{ccc} 0&x&0\\ \noalign{\medskip}x&0&0
\\ \noalign{\medskip}0&0&\frac{1}{x^2}\end {array} \right],\sigma_2\rightarrow\left[ \begin {array}{ccc} \frac{1}{x^2}&0&0\\ \noalign{\medskip}0&0&{x}^{
2}\\ \noalign{\medskip}0&1&0\end {array} \right]
.$$
\item What are the images of the Gaussian representations?  Goldschmidt and Jones \cite{GJ} computed the braid group images in $ES(p,n-1)$ for $p$ an odd prime; they are (essentially) symplectic groups over the field with $p$ elements. For $p=2$ the images are extensions of extra-special $2$ groups by symmetric groups (\cite{FRW}).  By Lemma \ref{primelemma}, to understand the general Gaussian representation images it is enough to compute them for prime powers, \emph{i.e.} in $ES(p^{k},n-1)$.

\item Turaev \cite{Tur} introduced \emph{enhanced Yang-Baxter operators} in order to produce link invariants from BVSs. What is the computational complexity of evaluating (approximately) the link invariants coming from locally monomial and Gaussian BVSs?  For Gaussian BVSs at primes the invariants are computed in \cite{GJ}.  For related work, see \cite{HNW}.

\end{enumerate}

\thispagestyle{empty}

\end{document}